\documentclass[12pt]{amsart}
\usepackage{amssymb,amsmath,amsthm,newlfont,enumerate}

\usepackage{hyperref}
\hypersetup{
    colorlinks=true,
    linkcolor=red,     
    urlcolor=pink,
    citecolor=blue
}
\usepackage[margin=2.5cm]{geometry}
\usepackage{enumerate}

\theoremstyle{plain}
\newtheorem{theorem}{Theorem}[section]
\newtheorem*{theorem*}{Theorem}
\newtheorem{proposition}[theorem]{Proposition}

\newtheorem{corollary}[theorem]{Corollary}

\theoremstyle{definition}
\newtheorem{definition}[theorem]{Definition}
\newtheorem{example}[theorem]{Example}

\theoremstyle{remark}

\numberwithin{equation}{section}

\newcommand{\bC}{\mathbb{C}}

\newcommand{\bN}{\mathbb{N}}

\newcommand{\bD}{\mathbb{D}}

\newcommand{\cP}{\mathcal{P}}

\newcommand{\cB}{\mathcal{B}}

\newcommand{\cA}{\mathcal{A}}

\newcommand{\cD}{\mathcal{D}}
\newcommand{\cT}{\mathcal{T}}

\newcommand{\cF}{\mathcal{F}}

\newcommand{\ra}{\rightarrow}

\DeclareMathOperator{\hol}{Hol}

\begin{document}

\date{\today}
\title[Kernel Summability Methods]{Kernel-Summability Methods and the Silverman-Toeplitz Theorem}

\author{Pierre-Olivier Paris\'e}
\address{Department of Mathematics, University of Hawai'i at Manoa,
Honolulu, Hawai'i,  United-States, 96822.}
\email{parisepo@hawaii.edu}

\thanks{POP is supported by NSERC and FRQNT postdoctoral scholarships.}

\begin{abstract}
We introduce kernel-summability methods in Banach spaces using the vector-valued integrals and prove an analogue of the Silverman-Toeplitz Theorem for regular kernel-summability methods. We also show that if $X$ is a Banach space and one kernel-summability method is included in another kernel-summability method for scalar-valued functions, then the first method is included in the second method, for $X$-valued functions. This extends a previous result from Javad Mashreghi, Thomas Ransford and the author. We then apply these abstract results to the summability of Taylor series of functions in a Banach space of holomorphic functions on the unit disk.
\end{abstract}

 % 40C15,


\maketitle

\section{Introduction}\label{Sec:Intro}
Given a (complex) Banach space $X$, a summability method is a triplet $(A, c_A (X), \lim_A )$, where $A$ is an application, $c_A (X)$ is called the summability domain, and $\lim_A$ is the new operation of limit induced by the application $A$. For instance, in the case of a matrix-summability method with $X = \bC$, the application $A$ would be determined by an infinite matrix $(a_{m, n})_{m, n \geq 0}$ of complex numbers such that a sequence $(v_n)_{n \geq 0}$ is transformed into the new sequence
	\begin{align*}
	A (v_n)_{n \geq 0} := \Big( \sum_{n \geq 0} a_{m, n} v_n \Big)_{m \geq 0} .
	\end{align*}
	We then say that a sequence $(v_n)$ is $\cA$\textit{-summable} if $A(v_n)_{n \geq 0}$ converges. The book by Boos \cite{Boos2000} is a good presentation of the modern point of view on summability theory. In Section \ref{Sec:GenTermSumTheory}, we introduce in more details the terminology attached to this modern point of view of the theory of summability.
	
An important and classical result in the theory of summability is the famous Silverman-Toeplitz Theorem (see \cite[Theorem 1, Theorem 2]{hardy1949}). This result gives necessary and sufficient conditions for a matrix-summability method to preserve limits of convergent sequences of complex numbers. The matrix-summability methods that preserve limits of convergent sequences are called \textit{regular}. The precise statement of the theorem is presented below.
\begin{theorem*}[Silverman-Toeplitz]
A matrix-summability method $\cA = (A, c_A (X), \lim_A )$, with $A = (a_{m, n})_{m, n \geq 0}$ and $a_{m, n} \in \bC$, is regular if and only if the following conditions are satisfied:
	\begin{enumerate}
	\item $\sup_{m \geq 0} \sum_{n \geq 0} |a_{m, n}| < \infty$;
	\item $\lim_{m \ra \infty} a_{m, n} = 0$, for every integer $n \geq 0$;
	\item $\lim_{m \ra \infty} \sum_{n \geq 0} a_{m, n} = 1$.
	\end{enumerate}
\end{theorem*}
\noindent  The aim of this paper is to explore generalizations and variants of the Silverman-Toeplitz Theorem.

In Section \ref{Sec:SilvTopThm}, we will explore generalizations of the Silverman-Toeplitz Theorem adapted to vector-valued functions. The main result of this section is a precise statement of the Silverman-Toeplitz Theorem for kernel-summability methods. Kernel-summability methods are summability methods given by a (Borel) measurable function $a : F \times E \ra \bC$ from two topological spaces $E$ and $F$ into $\bC$. These methods are applied to functions $v : E \ra X$, where $X$ is a Banach space, in the following way:
	\begin{align}
	\int_{E} a (r, t) v (t) \, d\mu (t) , \label{Eq:DefIntegralKernelSummabilityMethods}
	\end{align}
where $\mu$ is a Borel measure on $E$. Since the function $v$ is vector-valued, in Section \ref{Sec:VecValueInt}, we give an introduction to vector-valued integrals, which is based on \cite{Lewis2022}, to make sense of \eqref{Eq:DefIntegralKernelSummabilityMethods}. Therefore, the idea behind the kernel-summability methods is to transform possibly divergent integrals into convergent integrals. Notice that when $E = F = \bN$, the natural numbers including $0$, then we recover the matrix-summability methods. 

The Silverman-Toeplitz theorem is a result on inclusion of summability methods. Indeed, let $\mathcal{I}$ denote the summability method of strong convergence. This means that the operator in $\mathcal{I}$ is given by the identity $I (v) = v$, where $v : E \ra X$ and $\lim_I$ is the standard convergence of sequences. Then the Silverman-Toeplitz Theorem tells you that the summability method of strong convergence is included in any regular kernel-summability method, meaning that if a function $v : E \ra X$ converges to some vector $x \in X$, then the integral in \eqref{Eq:DefIntegralKernelSummabilityMethods} also converges to $x$. By adopting this point of view, we give a result of the same flavor of the Silverman-Toeplitz Theorem in Section \ref{Sec:InclMethods} by replacing the summability of strong convergence by an arbitrary kernel-summability method. Loosely speaking (see Theorem \ref{Thm:inclusionResult} below for the precise statement), we show that if one kernel-summability method $\cA$ is included in another kernel-summability method $\cB$ only when restricted to complex-valued functions, then $\cA$ is also included in $\cB$ when restricted to vector-valued functions. This extends Theorem 5.1 in \cite{MashreghiPariseRansford2021} to kernel-summability methods. The statement of Theorem \ref{Thm:inclusionResult} involves operators from a Banach space $X$ into another Banach space $Y$. Using operators enable the use of the Banach-Steinhauss Theorem in the proof. Is it possible to remove these operators in the statement of this theorem and still obtain the same conclusion? We partially answer this question in the context of reflexive Banach spaces by introducing the concept of \textit{weakly}-inclusion of summability methods. 

Originally, the inspiration behind the statement of Theorem \ref{Thm:inclusionResult} comes from the applications of summability theory to approximations in Banach spaces of holomorphic functions that were studied in \cite[Section 6]{MashreghiPariseRansford2021}. With this original motivation in mind, we present in Section \ref{Sec:ConsHolBanach} general consequences of the results from Section \ref{Sec:InclMethods} on the summability of Taylor series in Banach spaces of holomorphic functions in the unit disk.

\section{Vector-Valued Integrals}\label{Sec:VecValueInt}
Let $(E , \cT_E )$ be a topological space. Throughout the paper, we assume that $E$ is locally compact and Hausdorff. We will also always assume that $E$ is not compact. Let $X$ be a Banach space. The space of functions from $E$ into $X$ is denoted by $\cF (E, X)$. It is a complex vector space with the usual addition and scalar multiplication of functions. 

We also equip $E$ with a Radon measure $\mu$ defined on the family of Borel sets, meaning that $\mu (K) < \infty$ for any compact set $K \subset E$. We will also suppose that $E$ is $\sigma$-compact. As a consequence, we can find a sequence of compact sets $(K_n)_{n \geq 0}$ such that $K_n \subsetneq K_{n+1}$ and $\cup_n K_n = E$. In particular, this means that $E$ is $\sigma$-finite. We denote by $L^1_{\bC} (E)$ the space of integrable complex-valued functions on $E$ and by $L^\infty_{\bC} (E)$ the space of essentially bounded complex-valued functions on $E$. This section is based on a work by Lewis \cite{Lewis2022}.

\subsection{Measurability}
A function $s \in \cF (E, X)$ is a \textit{step function} if there are pairwise disjoint measurable sets $E_1$, $E_2$, $\ldots$, $E_n$ of finite measure and vectors $x_1$, $x_2$, $\ldots$, $x_n$ in $X$ such that
	\begin{align*}
	s = \sum_{j =1}^n x_j \chi_{E_j},
	\end{align*}
where $\chi_{E_j}$ is the characteristic function of $E_j$. 

\begin{definition}
A function $v \in \cF (E, X)$ is \textit{strongly measurable} if there exists a sequence of step functions $(s_n)_{n \geq 0}$ and a measurable set $Z$ of measure zero such that $\Vert s_n (t) - v(t) \Vert_X \ra 0$ as $n \ra \infty$ for any $t \in E \backslash Z$.
\end{definition}
An important fact we will use later is the following: If $(v_n)_{n \geq 0}$ is a sequence of strongly measurable functions and if $v$ is a function such that $v_n(t) \ra v(t)$ as $n \ra \infty$ $\mu$-a.e. on $E$, then $v$ is strongly measurable. 

\begin{definition}
A function $v \in \cF (E , X)$ is \textit{weakly measurable} if for any continuous linear functional $\phi \in X^\ast$, the function $\phi \circ v : E \ra \bC$ is measurable.  
\end{definition}
There is an important result connecting weak measurability to strong measurability. To state this connection, we need the notion of separably valued functions. A function $v : E \ra X$ is \textit{essentially separably valued} if there exists a measurable set $Z$ of measure zero and there is a countable set $C \subset X$ such that $f (E \backslash Z ) \subset \overline{C}$.
		\begin{theorem}[Pettis Measurability Theorem]
		A function $v : E \ra X$ is strongly measurable if, and only if, the two following conditions are satisfied:
			\begin{enumerate}
			\item $v$ is weakly measurable;
			\item $v$ is essentially separably valued.
			\end{enumerate}
		\end{theorem}
	In particular, if $v$ is strongly measurable, then it is weakly measurable. We now prove, as a consequence of the Pettis Measurability Theorem, the strong measurability of continuous functions, which is interesting on its own right. We present it as a proposition for reference later.
		\begin{proposition}\label{Prop:ContinuousImpliesStronglyMeas}
		If $v : E \ra X$ is continous, then it is strongly measurable.
		\end{proposition}
		\begin{proof}
		Notice first that $v$ must be weakly measurable. This comes from the fact that $\phi \circ v : E \ra \bC$ is a continuous function for every $\phi \in X^\ast$ and therefore measurable.  Furthermore, if $K \subset E$ is a compact set, then $v(K)$ is (essentially) separably valued. To see this, for any given $m$, cover $v(K)$ by a finite union of balls of radius $1/m$ and centered at $x_{n}^m$, with $n = 1, 2, \ldots N_m$. Then, let $C_m := \{ x_1^m , \ldots , x_{N_m}^m \}$ and take $C = \cup_{m \geq 1} C_m$. The restriction $\left. v \right|_{K} : K \ra X$ is strongly measurable by the Pettis Measurability Theorem. Since we assumed that $E = \cup_n K_n$, where $(K_n)_{n \geq 0}$ is an increasing sequence of compact sets, we see that $\left. v \right|_{K_n} \ra v$ pointwise on $E$. Therefore, the function $v$ is strongly measurable.
		\end{proof}

\subsection{Bochner Integral}\label{SecSub:BochnerIntegral}
The integral of a step function $s = \sum_{j= 1}^n x_j \chi_{E_j}$ is defined by
	\begin{align*}
	\int_E s (t) \, d\mu (t) := \sum_{j = 1}^n \mu (E_j) x_j .
	\end{align*}
It does not depend on the representation of $s$. A strongly measurable function $v \in \cF (E, X)$ is \textit{Bochner-integrable} if there is a sequence of step functions $(s_n)_{n\geq 0}$ and a measurable set $Z \subset E$ of measure zero such that
	\begin{enumerate}
	\item  $s_n(t) \ra v(t)$ as $n \ra \infty$, for any $t \in E \backslash Z$.
	\item  $t \mapsto \Vert v (t) - s_n (t) \Vert_X \in L^1_\bC (E)$ for any integer $n \geq 1$ and
		\begin{align*}
		\lim_{n \ra \infty} \int_E \Vert v (t) - s_n (t) \Vert_X \, d\mu = 0.
		\end{align*}
	\item the limit $\lim_{n \ra \infty} \displaystyle\int_A s_n (t) \, d\mu$ exists for any Borel set $A$.
	\end{enumerate}
The \textit{Bochner integral} of a Bochner integrable function $v$ is then defined by
	\begin{align*}
	\int_E v (t) \, d\mu (t) := \lim_{n \ra \infty} \int_E s_n (t) \, d\mu (t).
	\end{align*}
We denote by $L^1 (E, X)$ the space of Bochner integrable functions on $E$. The Bochner Integral satisfies the following properties:
	\begin{enumerate}[(I)]
	\item It is unique and linear;
	\item\label{P:BochnerEquivalentNormIntegrable} A strongly measurable function $v \in \cF (E, X)$ is Bochner integrable if and only if $t \mapsto \Vert v (t) \Vert_X \in L^1_\bC (E)$;
	\item\label{P:UpperBoundForBochnerIntegral} For any $v \in L^1 (E, X)$, we have $\displaystyle \Big\Vert \int_E v \, d\mu \Big\Vert_X \leq \int_E \Vert v \Vert_X \, d\mu$.
	\item\label{P:InterchangeBoundedOperator} If $T : X \ra Y$ is a bounded linear operator, where $Y$ is a Banach space, and $v \in L^1 (E, X)$, then $t \mapsto T(v(t)) \in L^1 (E, Y)$ and
		\begin{align*}
		T \Big( \int_E v (t) \, d\mu (t) \Big) = \int_E T(v(t)) \, d\mu (t) .
		\end{align*}
	\item\label{P:BochnerIntegralConstantFunction} If $f : E \ra \bC$ is in $L^1_{\bC} (E)$ and if $v(t) = f(t) x$ with $x \in X$, then $v$ is Bochner integrable and
		\begin{align*}
		\int_E v(t) \, d\mu (t) = \Big( \int_E f(t) \Big) x .
		\end{align*}
	\end{enumerate}
	We will also make use of the space $L^\infty (E, X)$ of essentially bounded vector-valued functions on $E$.
	
	\subsection{Pettis Integral}
	We introduce a weaker notion of integral, called in the literature the \textit{Pettis integral}. For a weakly measurable function $v: E \ra X$, we say that $v$ is Pettis integrable if 
		\begin{enumerate}
		\item the function $\phi \circ v \in L^1_{\bC} (E)$, for any continuous linear functional $\phi : X \ra \bC$;
		\item  for any Borel set $E' \subset E$, there exists a vector $I_{v, E'} \in X$ such that $\phi ( I_{v, E'} ) = \displaystyle \int_{E'} \phi (v) \, d\mu$, for any continuous linear functional $\phi : X \ra \bC$.
		\end{enumerate}
	In this case, the vector $I_v := I_{v, E}$ is the \textit{weak integral} of $v$ on $E$ and we also write, when the context is clear,
		\begin{align*}
		I_v = \int_E v (t) \, d\mu (t) .
		\end{align*}
	We also say that $v : E \ra X$ is \textit{weakly integrable} when the weak integral $I_v$ exists. The weak integral enjoys the following properties:
		\begin{enumerate}[(I)]
		\item It is unique and linear, when it exists;
		\item If $T : X \ra Y$ is a bounded linear operator, where $Y$ is a Banach space, and if $v$ is a weakly integrable function, then $t \mapsto T(v(t))$ is weakly integrable and
			\begin{align*}
			T \Big( \int_E v(t) \, d\mu (t) \Big) = \int_E T(v(t)) \, d\mu (t) .
			\end{align*}
		\item If $v$ is Bochner integrable, then $v$ is weakly integrable and the Bochner integral is equal to the weak integral. 
		\end{enumerate}
	
	\section{What Is A Summability Method?}\label{Sec:GenTermSumTheory}
The main goal of summability theory is to transform a sequence into another sequence which may have more chance to converge. In this section, we will extend this notion to transform functions into other functions which, hopefully, behave better than the original function. The notions and point of view adopted in this section are adapted from Boos' book \cite{Boos2000}. We let $X$ and $Y$ be Banach spaces with norm $\Vert \cdot \Vert_X$ and $\Vert \cdot \Vert_Y$ respectively. Let $(E, \cT_E)$, $(F, \cT_F)$ be two non-compact, locally compact and Hausdorff topological spaces.

\subsection{Convergence at Infinity}
To mimic the notion of convergence of sequences, we introduce the notion of \textit{convergence at infinity}. A function $v \in \cF (E, X)$ \textit{converges at infinity} if there is an element $x \in X$ such that for any $\varepsilon > 0$, there is a compact set $K \subset E$ such that 
	\begin{align*}
	\Vert v(t) - x \Vert_X < \varepsilon \quad (t \not\in K ). 
	\end{align*}
We denote this by $\lim_{t \ra \infty} v(t) = x$. The fact that $E$ is not compact makes this definition non void.
 
The set of all strongly measurable functions $v : E \ra X$ that are locally bounded and converge at infinity is denoted by $c (E, X)$. A function $v \in \cF (E, X)$ is \textit{locally bounded} if for any compact set $K \subset E$, $\sup_{t \in K} \Vert v(t) \Vert_X < \infty$. We see immediately that $c (E, X) \subset L^\infty (E, X)$. 
	
Here are two examples to illustrate the general theory of summability outlined in the next section.
	\begin{example}
	Let $E = \bN := \{ 0, 1, 2, \ldots \}$ endowed with the topology $\cT = \cP (\bN )$. In this case, since all functions $v : \bN \ra X$ are measurable, the set $c (E, X)$ is the set of all convergent sequences on $\bN$ and the set $L^\infty (E, X)$ is the set of all bounded sequences on $\bN$. The usual notation for these sets are $c (X)$ and $\ell^\infty (X)$.
	\end{example}
	
	\begin{example}\label{Example:IntervalSpace}
	Let $E = [0, R )$, with $0 < R \leq \infty$, endowed with the topology generated by the intervals of the form $(a, b)$, $[0, b)$. In this case, the set $\cF (E, X)$ is the set of functions $v : [0, R) \ra X$. The space $c (E, X)$ is the set of all strongly measurable functions on $E$ which have a limit at $R$. The set $L^\infty (E, X)$ is the set of essentially bounded functions on $E$. 
	\end{example}

\subsection{Definition of a Summability Method}
We define a summability method in the following way.
\begin{definition}
A \textit{summability method} is a triplet $(A, c_A (X) , \lim_A )$, where
	\begin{enumerate}
	\item $A : D_A \ra \cF (F, Y)$ is an application, with $D_A \subseteq \cF (E, X)$ a vector subspace called the \textit{domain} of $A$.
	\item $c_A (X)$ is the \textit{summability domain} defined by
		\begin{align*}
		c_A (X) := \big\{ v \in \cF (E, X) \, : \, v \in D_A \text{ and } \lim_{r \ra \infty} A(v) (r) \text{ exists} \big\} .
		\end{align*}
	\item $\lim_A : c_A (X) \ra X$ is the \textit{summability operator} defined by $\lim_A := \left. \lim \circ A \right|_{c_A (X)}$, where $\lim : c (E, X) \ra X$ is the limit operator at infinity.
	\end{enumerate}
\end{definition}
\noindent A function $v \in \cF (E, X)$ is \textit{$\cA$-summable} (or summable by the method $\cA$) if $v \in c_A (X)$. When the application $A$ is linear on $D_A$, the summability operator is a linear application. We will assume for the rest of the paper that $A$ is linear. Some authors call the application $\lim_A$ an \textit{operator of limits} (see, for example, \cite{Kangro1976, Leon2020}). 

We now give a list of classical summability methods to show how they are defined using the above definition. We will also introduce three types of summability methods that will be our main focus in this paper.

\subsection{Matrix-summability methods}

	\begin{example}
	Let $E = F = \bN$ and $\cT_E = \cT_F := \cP (\bN )$. Then we denote an element $v \in \cF (\bN , X )$ as $(v_n)_{n \geq 0}$. We define the application $I : \cF (\bN , X) \ra \cF (\bN , X)$ by
		\begin{align*}
		I(v) := (v_n)_{n \geq 0} \quad (v := (v_n)_{n \geq 0} \in \cF (\bN , X) )
		\end{align*}
	We have $D_I = \cF (\bN , X)$ and 
		\begin{align*}
		c_I (X) = \big\{ (v_n)_{n \geq 0} \, : \, \lim_{n \ra \infty} v_n \text{ exists} \big\} = c (X)
		\end{align*}
	is the space of all convergent sequences in $X$. The triplet $(I, c_I (X), \lim_I )$ is the \textit{strong convergence} in $X$.
	\end{example}
	
	\begin{example}
	Let $E$ and $F$ be as in the previous example. We define the application $C^0 : \cF (\bN , X) \ra \cF (\bN , X )$ by
		\begin{align*}
		C^0 (v) := \big( v_0 +  v_1 + \cdots + v_n \big)_{n \geq 0} \quad \big( v := (v_n)_{n \geq 0} \in \cF (\bN, X) \big) .
		\end{align*}
	We have $D_{C^0} (X) = \cF (\bN , X)$, the subspace $c_{C_0} (X)$ is the space of convergent series and the summability operator $\lim_{C^0}$ is the value of the convergent series $\sum_{n \geq 0} x_n$. The summability method $(C^0 , c_{C^0} (X), \lim_{C^0})$ is called the \textit{method of summation of series}. We denote the summability domain $c_{C^0} (X)$ by $cs (X)$:
		\begin{align*}
		cs (X) := \Big\{ (v_n)_{n \geq 0} \, : \, \sum_{n \geq 0} v_n \text{ converges} \Big\} .
		\end{align*}
	We usually use the notation $(s_n)_{n \geq 0}$ for the sequence $C^0 (v)$ and each $s_n$ is called the $n$-partial sums of the series $\sum_{n \geq 0} x_n$. 
	\end{example}
	
	\begin{example}
	Let $E$ and $F$ be as in the previous examples. We define the application $C : \cF (\bN , X ) \ra \cF (\bN , X)$ by
		\begin{align*}
		C (v) := \Big( \frac{v_0 + v_1 + \ldots + v_n}{n + 1} \Big)_{n \geq 0} \quad \big( v := (v_n)_{n \geq 0} \in \cF (\bN , X) \big) .
		\end{align*}
	Each member of the sequence $C (v)$ is called the \textit{Ces{\`a}ro mean} (of order $1$) of the sequence $v$ and is denoted by $(\sigma_n )_{n \geq 0}$. We have $D_{C} = X^{\bN}$ and
		\begin{align*}
		c_C (X) = \big\{ (v_n)_{n \geq 0} \, : \, \lim \circ C (v_n)_{n \geq 0} = \lim_{n \ra \infty} \sigma_n \text{ exists} \big\} .
		\end{align*}
	The triplet $(C , c_C (X), \lim_C)$ is called the \textit{Ces{\`a}ro summability method of order $1$} or simply the \textit{Ces{\`a}ro method}. 
	\end{example}
	There is a more general definition of the Ces{\`a}ro method involving the generalized binomial coefficients. We will however not mention its definition here. We would rather only mention a surprising fact about the summability domain of the generalized Ces{\`a}ro methods. It was only recently that a complete characterization of the summability domain, when $X = \bC$, of the Ces{\`a}ro summability methods was obtained \cite{leonetti2020}. This shows that, in general, it is hard to describe explicitly the summability domain of the application $A$ in a summability method. 
	
	The previous examples regroup a special type of summability methods. They can be seen as multiplying an infinite matrix with the infinite sequence $(v_n)$. For example, the application $C$ in the Ces{\`a}ro summability method can be rewritten as
		\begin{align*}
		C(v) = \begin{pmatrix}
		1 & 0 & 0 & \cdots \\
		1/2 & 1/2 & 0 & \cdots \\
		1/3 & 1/3 & 1/3 & \cdots \\
		\vdots & \vdots & \vdots & \ddots
		\end{pmatrix} \begin{pmatrix}
		v_0 \\ v_1 \\ v_2 \\ \vdots
		\end{pmatrix} .
		\end{align*}
	This inspires the following definition.
	\begin{definition}
	Let $E = F = \bN$. Then a summability method $(A, c_A (X) , \lim_A )$ is called a \textit{matrix-summability method} if there exist scalars $a_{m, n} \in \bC$, $m, n \in \bN$, such that the application $A$ is given by
		\begin{align*}
		A (v) = \Big( \sum_{n \geq 0} a_{m, n} v_n \Big)_{n \geq 0} \quad \big( v := (v_n)_ {n \geq 0} \in D_A \big) .
		\end{align*}
	\end{definition}
	\noindent In this case, we let $A = (a_{m, n})_{m, n \geq 0}$ be the infinite matrix with coefficients $a_{m, n}$. Hardy's book \textit{Divergent series} \cite{hardy1949} is the classical reference for matrix-summability methods.
	
	\subsection{Sequence-to-function summability methods}
	Alongside the Ces{\`a}ro summability method, one of the well-known summability methods is the Abel summability method.
	
	\begin{example}
	Let $E = \bN$ with $\cT_E = \cP (\bN )$ and $F = [0, 1)$ with the topology $\cT_F$  generated by the sets of the form $[0, b)$ and $(a, b)$ ($0 < a \leq b < 1$). Notice here that the limit at infinity becomes the limit at $1$. For a sequence $v := (v_n)_{n \geq 0}$, we define the application $A^0_v : [0, r) \ra X$ by
		\begin{align*}
		A^0_v (r) := \sum_{n = 0}^\infty (1 - r)r^n v_n \quad (0 \leq r < 1 )
		\end{align*}
	if the series converges in $X$. The expression of $A^0_v (r)$ is called the \textit{Abel mean} (of order $0$) of the sequence $v$. We then define the application $A : D_{A^0} \subseteq \cF (\bN , X) \ra \cF ([0, 1) , X)$ by 
		\begin{align*}
		A^0 (v) := A^0_v \quad (v \in D_{A^0}) .
		\end{align*}
	In general, we have $D_{A^0} \subsetneq X^{\bN}$ and
		\begin{align*}
		c_{A^0} (X) = \big\{ v := (v_n)_{n \geq 0} \, : \, \lim\circ A^0 (v) = \lim_{r \ra 1^-} \sum_{n = 0}^\infty (1 - r)r^n v_n \text{ exists} \big\} .
		\end{align*}
	The triplet $(A^0 , c_{A^0} (X) , \lim_{A^0})$ is called the \textit{Abel summability method} (of order $0$). 
	\end{example}
	
	Abel summability method is an example of a sequence-to-function summability method.
	\begin{definition}
	Let $E = \bN$ and $(F, \cT_F )$ be a locally compact Hausdorff topological space which is non compact. A summability method $(A, c_A (X), \lim_A )$ is called a \textit{sequence-to-function} summability method if there is a sequence $(a_n)_{n \geq 0}$ of measurable functions $a_n : F \ra \bC$ such that
		\begin{align*}
		A_v (r) = \sum_{n \geq 0} a_n (r) v_n \quad \big( v := (v_n)_{n \geq 0} \in D_A, \, r \in F \big) .
		\end{align*}
	\end{definition}

	\subsection{Kernel summability methods}	
	Sequence-to-function summability methods were recently used in \cite{MashreghiPariseRansford2021} to study the summability of Taylor series of holomorphic functions in the unit disk. One of them can be seen as a special type of kernel summability method.
	
	\begin{example}
	Let $E = F = [0, 1)$ and $\cT_E = \cT_F$ be the topology in the last example. We define the application $L : D_L \subseteq \cF ([0, 1) , X) \ra \cF ([0, 1), X)$ by
		\begin{align*}
		L (v) (r) := \frac{-1}{\log (1 - r)} \int_0^r \frac{v(t)}{1 - t} \, dt ,
		\end{align*}
	where $v \in D_L$. The subspace $D_L$ is defined as
		\begin{align*}
		D_L = \big\{ v \in \cF ([0, 1) , X) \, : \, L(v) (r) \text{ exists for any } r \in [0, 1) \big\} 
		\end{align*}	
	and the summability domain is
		\begin{align*}
		c_L (X) = \big\{ v \in \cF ([0, 1) , X) \, : \, \lim_{r \ra 1^-} L(v) (r) \text{ exists} \big\} .
		\end{align*}
	The triplet $(L , c_L (X) , \lim_L )$ is called the \textit{logarithmic summability method}. This summability method was introduced by Borwein \cite{Borwein1957a}.			 
	\end{example}
	\noindent Inspired by the previous example, we introduce the following type of summability methods.
	\begin{definition}
	Let $(E, \cT_E )$ and $(F , \cT_F )$ be two locally compact Hausdorff topological spaces which are non compact. Let $\mu$ be a Borel measure on $E$. A summability method $(A, c_A (X) , \lim_A )$ is called a \textit{kernel-summability method} if there exists a (Borel) measurable function $a : F \times E \ra \bC$ such that the application $A : D_A \subset \cF (E, X) \ra \cF (F, X)$ is given by
		\begin{align*}
		A(v) (r) := \int_E a (r, t) v(t) \, d\mu (t) , 
		\end{align*}
	where the weak integral exists for any $v \in D_A$ and $r \in F$. The function $a : F \times E \ra \bC$ is called the \textit{summability kernel}.
	\end{definition}
	
	We still need to introduce two more concepts to see where the Silverman-Toeplitz Theorem sits in the general theory of summability.
	
	\subsection{Conservative summability methods}
	As the Ces{\`a}ro summability method has shown, it is hard to describe explicitly the domain of the application in a summability method. We can relax the goal and instead search for subspaces that are contained in the domain $D_A$. We can also try to describe the image $A(M)$ for some subspace $M \subseteq D_A$. To make these last statements more precise, we introduce the following concept.
	\begin{definition}
	Let $\cA := (A, c_A (X), \lim_A )$ be a kernel-summability method. Let $M \subseteq \cF (E, X)$ and $N \subseteq \cF (F, X)$ be two (vector) subspaces. The method $\cA$ is called \textit{$(M, N)$-conservative} if $M \subseteq D_A$ and $A (M) \subseteq N$. We denote by $(M, N)$ the class of $(M , N)$-conservative summability method.
	\end{definition}
	The first chapter of Maddox's memoir \cite{Maddox1980} has an exhaustive summary of the problem of describing the family $(M, N )$ of matrix-summability methods for various choices of spaces $M$ and $N$. The ones that are closely related to the Silverman-Toeplitz Theorem are the classes of bounded summability methods, and of conservative summability methods. In the case $X = \bC$, there are usually defined as followed, :
		\begin{enumerate}
		\item $\cA$ is \textit{bounded} if $A \in ( \ell^\infty (\bC ) , \ell^\infty (\bC ))$.
		\item $\cA$ is \textit{conservative} if $A \in (c(X), c(Y))$.
		\end{enumerate}
	These classes can be defined without the notion of measurability. For kernel-summability methods, the above definitions must be adapted to incorporate measurable functions.
	
	\begin{definition}
	Let $\cA = (A, c_A (X), \lim_A )$ be a kernel-summability method. 
		\begin{enumerate}
		\item $\cA$ is \textit{bounded} if $A \in (L^\infty (E, X) , L^\infty (F, X))$.
		\item $\cA$ is \textit{conservative} if $A \in (c (E, X) , c(F, X))$.
		\end{enumerate}
	\end{definition}
	Since $c (E, X) \subset L^\infty (E, X)$, we have that $(c (E, X) , c (F, Y)) \subset (L^\infty (E, X) , L^\infty (F, X))$. The strong convergence method, the Ces{\`a}ro method, the Abel method and the logarithmic method are all conservative. The method of summation of series is however not conservative, nor bounded. Indeed, taking $x_n = x$, with $\Vert x \Vert_X = 1$, we have $s_n = n x$ which is not bounded. We will mainly use these classes of summability methods to define regularity and scalar-regularity.
	
	\section{The Silverman-Toeplitz Theorem}\label{Sec:SilvTopThm}
	From Ces{\`a}ro Lemma, we have that
		\begin{align*}
		c (X) \subset c_C (X) \quad \text{ and } \quad \lim_{n \ra \infty} \sigma_n = \lim_{n \ra \infty} v_n \quad \big( (v_n)_{n \geq 0} \in c (X) \big) .
		\end{align*}
	Also, from Abel's Theorem, we know that
		\begin{align*}
		c (X) \subset c_{A^0} (X) \quad \text{ and } \quad \lim_{r \ra 1^-} A_r^0 (v_n)_{n \geq 0} = \lim_{n \ra \infty} v_n \quad \big( (v_n)_{n \geq 0} \in c ( X) \big) .
		\end{align*}
	We say that the Ces{\`a}ro and Abel methods are \textit{regular}.
	\begin{definition}
	Let $\cA := (A , c_A (X) , \lim_A )$ be a kernel-summability method. We say that $\cA$ is \textit{regular} if
		\begin{enumerate}
		\item the method $\cA$ is conservative;
		\item the method $\cA$ preserves the value of limits, that is, for any $v \in c (E, X)$, we have
			\begin{align*}
			\lim_A v(t) = \lim_{r \ra \infty} A(v) (r) = \lim_{t \ra \infty} v(t) .
			\end{align*}
		\end{enumerate}
	\end{definition}
	
	In the introduction, we stated the conditions for a matrix-summability method to be regular in the statement of the Silverman-Toeplitz Theorem. We will now present generalizations of this theorem starting with matrix-summability methods applied to vector-valued sequences.
	
	\subsection{Matrix-summability methods}
	The Silverman-Toeplitz theorem can be extended to matrix-summability methods applied to sequences of vectors in a Banach space $X$. The conditions are necessary comes from the fact that $\bC$ embeds isometrically in $X$ via the transformation $\lambda \mapsto \lambda x$, for a unit vector $x$. Therefore, the summability method is regular for \textit{scalar}-valued sequences and the Silverman Toeplitz theorem can be applied to conclude that the conditions should be necessary. On the other hand, the fact that the conditions are sufficient comes from elementary estimates. A more general proof for kernel-summability methods will be given in the next subsection.
	
	The Silverman-Toeplitz Theorem was also extended by Robinson (\cite[Theorem VII]{Robinson1950}) to matrix-summability method where the entries are linear (not necessarily bounded) operators $A_{n, k} : X \ra Y$, where $X$ and $Y$ are Banach spaces. In the statement of the next Theorem, the \textit{group norm} of a row of linear operators $T = (T_k)_{k \geq 0}$ is defined by
		\begin{align*}
		\Vert T \Vert := \sup \Big\{ \Big\Vert \sum_{k = 0}^n T_k x_k \Big\Vert_Y \, : \, n \in \bN \text{ and } \Vert x_k \Vert_X \leq 1 \Big\} .
		\end{align*}	
	Also, the acronym ``SOT'' means the strong operator topology.
		 
	\begin{theorem}[Robinson]
	Let $(A , c_A (X) , \lim_A )$ be a matrix-summability method where the matrix is given by $(A_{m, n})_{m, n \geq 0}$, with $A_{n, k} : X \ra Y$ linear operators. The method $\cA$ is regular if and only if there exists an integer $N \geq 0$ such that
		\begin{enumerate}
		\item $\sup_{m \geq 0} \Vert (A_{m, n})_{k \geq N} \Vert < \infty$;
		\item $\lim_{n \ra \infty} A_{m, n} = 0$ (SOT);
		\item $\sum_{n = 0}^\infty A_{m, n} = A_m$ exists (SOT) for any $n \geq 0$;
		\item $\lim_{m \ra \infty} A_m = I$ (SOT).
		\end{enumerate}
	\end{theorem}
	
	The group norm enjoys many useful properties (see \cite[Proposition 2.3]{Maddox1980}). One of them is the following: If $T_k = a_k I$ where $I : X \ra Y$ is the identity operator and $a_k \in \bC$, then $\Vert (T_k)_{k \geq 0} \Vert = \sum_{k \geq 0} |a_k|$.	The last property tells us that the group norm of a sequence of complex numbers is simply its $\ell^1$-norm. Because of this last property of the group norm, we see how condition (1) in Robinson's Theorem is the right generalization of condition (1) in the Silverman-Toeplitz Theorem. 
	
	Robinson proved this result using a ``gliding hump'' argument adapted from Hardy's book \textit{Divergent Series} \cite{hardy1949}. There is a functional analytic proof of Robinson's result in Maddox's book (see \cite[Theorem 4.2]{Maddox1980}). %It uses a Theorem of Lorentz and Macphail that be found in \cite{Lorentz1952} (or see \cite[Theorem 4.1]{Maddox1980}). 
	
	Recently, Leonneti proved a generalized version of Robinson's result. Instead of using the standard $\lim$ operator of sequences to construct the $\lim_A$ in the triplet $(A, c_A (X) , \lim_A )$, he replaced it by the notion of ideal-convergence. The word ``ideal'' refers to a family $\mathcal{I}$ of subsets of $\bN$ with desirable properties. For more details on this recent result, the reader may consult \cite{Leonetti2022}. Leonniti's proof uses a generalization of the Banach-Steinhauss Theorem using the notion of limit superior with respect to an ideal of nonnegative integers.
	
	There is a version of the Silverman-Toeplitz Theorem for sequence-to-function summability methods $(A, c_A (X), \lim_A )$ given by a sequence of functions $(a_n)$. For more details on that, see \cite[Theorem 5]{hardy1949} and \cite[Theorem 2.1]{MashreghiPariseRansford2021} when the functions $(a_n)_{n \geq 0}$ are scalar-valued and see \cite[Theorem X]{Robinson1950} when the functions $(a_n)_{n \geq 0}$ are operator-valued. 
	
	\subsection{Kernel Summability Methods}	
	Because of the different ways the kernel function $a : F \times E \ra \bC$ may behave compared to the more predictable behavior of a kernel $a : \bN \times E \ra \bC$ (matrix or sequence-to-function summability methods), Hardy \cite[p.50]{hardy1949} mentioned that the expected conditions in the Silverman-Toeplitz Theorem are less ``symmetric'' than that in the other versions of the Silverman-Toeplitz Theorem. He therefore confined himself to state sufficient conditions. He does, however, work out potential necessary and sufficient conditions later in his book \cite[pp.61-63]{hardy1949}. We present here an adaptation of these conditions for kernel-summability methods with kernel defined on general $E$ and $F$. Recall that there is a sequence of compact sets $(K_n)_{n \geq 0}$ such that $E = \cup_{n \geq 0} K_n$, $K_{n} \subsetneq K_{n + 1}$, and $\mu (K_n) < \infty$.
	\begin{theorem}\label{Thm:SilvermanToeplitzKernel}
	Let $\cA = (A, c_A (X), \lim_A )$ be a kernel-summability method with kernel $a : F \times E \ra \bC$. The method $\cA$ is regular if and only if
		\begin{enumerate}
		\item the map $t \mapsto a(r, t)$ is in $L^1_\bC (E)$ for any $r \in F$;
		\item the map $r \mapsto \int_E |a (r, t)| \, d\mu (t) \in L^\infty (F, \bC )$;
		\item for any compact set $K \subset E$, $\displaystyle\int_{K} |a(r, t)| \, d\mu (t) \ra 0$ as $r \ra \infty$;
		\item $\displaystyle\lim_{r \ra \infty} \int_{E} a (r, t) \, d\mu (t) = 1$.
		\end{enumerate}
	\end{theorem}
	 
	\begin{proof}
	Suppose that the conditions (1)--(4) are satisfied. We have to show that $\cA$ is regular. Therefore, we have to show the following:
		\begin{enumerate}[a)]
		\item $c (E, X) \subset D_A$;
		\item $A (c (E, X)) \subset c (F, X)$ and;
		\item $\lim_{r \ra \infty} A_v (r) = \lim_{t \ra \infty} v(t)$ for any $v \in c (E ,X)$.
		\end{enumerate}
		
	We start by showing (a). Let $v \in c (E, X)$. In particular, the function $v$ is strongly measurable and bounded. Therefore, the function $t \mapsto a(r, t) v(t)$ is strongly measurable and the Bochner integral
		\begin{align*}
		A_r (v) := A (v) (r) = \int_E a(r, t) v(t) \, d\mu (t)
		\end{align*}
	exists by the condition (1) above and by an application of Property \eqref{P:BochnerEquivalentNormIntegrable} in \S\ref{SecSub:BochnerIntegral}. In particular, $v \in D_A$.
	
	We now show (b), that is $A (c (E, X)) \subset c (F, X)$. By condition (2), the map $r \mapsto A_r (v)$ is strongly measurable. To show that $\lim_{r \ra \infty} A_r (v)$ exists, we will show directly that (c) holds. With the notation introduced above, we have to show that 
		$$
		\lim_{r \ra \infty} A_r (v) = \lim_{t \ra \infty} v(t),
		$$ 
	for any $v \in c_{\infty} (E, X)$. Let $v \in c_{\infty} (E, X)$ and let $x := \lim_{t \ra \infty} v(t)$. We have
		\begin{align*}
		\Vert A_r (v) - x \Vert_X &= \Big\Vert \int_E a(r, t) v(t) \, d\mu (t) - x \Big\Vert_X \\
		&= \Big\Vert \int_E a(r, t) (v(t) - x ) \, d\mu (t) + \Big( \int_E a(r, t) \, d\mu (t) - 1 \Big) x \Big\Vert_X .
		\end{align*}
	Let $\varepsilon > 0$. There is a compact set $K_0 \subset E$ such that $\Vert v(t) - x \Vert_X < \varepsilon$ for $t \not\in K_0$. Then, we have
		\begin{align*}
		\Big\Vert \int_E a(r, t) (v(t) - x ) \, d\mu (t) \Big\Vert_X & \leq \Big\Vert \int_{K_0} a(r, t) (v(t) - x ) \, d\mu (t) \Big\Vert_X \\
		& \qquad \qquad \qquad + \Big\Vert \int_{E\backslash K_0} a(r, t) (v (t) - x) \, d\mu (t) \Big\Vert_X \\
		& \leq \sup_{t \in K_0} \Vert v(t) - x \Vert_X \int_{K_0} |a(r, t)| \, d\mu (t) + M \varepsilon .
		\end{align*}		 
	Therefore, taking a compact set $K \subset F$ sufficiently big, we see that for any $r \not\in K$, we have
		\begin{align*}
		\Vert A_r (v) - x \Vert_X < \sup_{t \in K_0} \Vert v(t) - x \Vert_X \varepsilon + M \varepsilon + \varepsilon \Vert x \Vert_X .
		\end{align*}
	Therefore, $\lim_A v(t) = \lim v(t)$.
	
	We will now show that the four conditions are necessary. We first start by showing that condition (1) is necessary. Given a unit vector $x \in X$, the function $v(t) = x$, for every $t \in E$, belongs to $c (E, X)$. Therefore, since $c (E , X) \subset D_A$, the (weak) integral
		\begin{align*}
		\int_E a (r, t) x \, d\mu (t)
		\end{align*}
	exist, for every $r \in F$. However, for a fixed $r \in F$, using the Hahn-Banach Theorem, we can find a continuous linear functional $\phi \in X^\ast$ such that $\phi (x) = \Vert x \Vert_X = 1$. Therefore, from the definition of the weak integral, we see that $\phi (a(r, t) x) = a(r, t)$ must be in $L^1_{\bC} (E)$. Furthermore, we have $\lim_{t \ra \infty} v(t) = x$ and therefore, from the regularity of $\cA$,
		\begin{align*}
		\lim_{r \ra \infty} \int_E a(r, t) v(t) \, d\mu (t) = x.
		\end{align*}
	Using the fact that $\phi (x) = \Vert x \Vert_X = 1$, we conclude that
		\begin{align*}
		\lim_{r \ra \infty} \int_E a(r, t) \, d\mu (t) = 1 .
		\end{align*}
	Thus, condition (4) is also necessary. 
	
	Notice that since we now know that $t \mapsto a(r, t) \in L^1_{\bC} (E)$ for every $r \in F$, then the Bochner integral of $t \mapsto a(r, t) v(t)$ exists, for every $r \in F$ and every $v \in c (E, X)$.
	
	To show that condition (2) is necessary, we will introduce some notations. Let $\Vert \cdot \Vert_{\infty}$ be the essential supremum norm on $L^\infty (E, X)$. Therefore, after identifying the functions that agree a.e. on $E$, the space $L^\infty (E, X)$ becomes a Banach space and $c (E, X)$ becomes a Banach space under the same norm. We define the linear operator $A_r : c (E, X) \ra X$ as followed:
		\begin{align*}
		A_r (v) := \int_E a(r, t) v (t) \, d\mu (t) .
		\end{align*}
	For each $r \in F$, the  linear operator $A_r : c (E, X) \ra X$ is bounded. Indeed, given $v \in c (E, X)$ and using Property \eqref{P:UpperBoundForBochnerIntegral} in \S\ref{SecSub:BochnerIntegral},
		\begin{align*}
		\Vert A_r (v) \Vert_X \leq \int_E |a(r, t)| \Vert v(t) \Vert_X \, d\mu (t) \leq \Big( \int_E |a (r, t)| \, d\mu (t) \Big) \Vert v \Vert_{\infty} .
		\end{align*}
	For each $v \in c (E, X)$, the function $r \mapsto A_r (v)$ belongs to $c (F, X)$ because the method $\cA$ is regular. Therefore, for each $v \in c (E, X)$, we have $\sup_{r \in F} \Vert A_r (v) \Vert_X < \infty$. By the Banach-Steinhauss Theorem, we conclude that
		\begin{align*}
		\sup_{r \in F} \Vert A_r \Vert_{c (E, X) \ra X} < \infty .
		\end{align*}
	It remains to compute the norm of $A_r$. Define the bounded linear operator $A_r^n : c (E, X) \ra X$ as followed:
		\begin{align*}
		A_r^n (v) := \int_{K_n} a(r, t) v(t) \, d\mu (t) .
		\end{align*}
	Let $x \in X$ be a unit vector and let $E' := \{ t \in E \, : \, a(r, t) \neq 0 \}$. Define the function $v_n : E \ra X$ by 
	\begin{align}
	v_n := \chi_{E' \cap K_n} a(r, \cdot ) / \overline{a (r, \cdot )} . \label{Eq:FunctionTestCondition3Silverman}
	\end{align}
	Then we have $v_n \in c_\infty (E, X)$, $\Vert v_n \Vert_\infty = 1$, and moreover $v_n$ is Bochner integrable. Applying $A_r^n$ and using Property \eqref{P:BochnerIntegralConstantFunction} in \S\ref{SecSub:BochnerIntegral}, we obtain
		\begin{align*}
		 \Vert A_r^n (v_n) \Vert_X = \int_{K_n} |a(r, t)| \, d\mu (t) \Vert x \Vert_X = \int_{K_n} |a (r, t)| \, d\mu (t) .
		\end{align*}
	Therefore, the norm of $A_r^n$ is the right-hand side of the last equality. We then obtain
		\begin{align*}
		\Vert A_r^{n} (v) - A_r (v) \Vert_X \leq \int_{E \backslash K_n} |a(r, t)| \, d\mu (t) \quad (\Vert v \Vert_{\infty} \leq 1 )
		\end{align*}
	and using the Lebesgue Dominated Convergence Theorem in $L^1_{\bC} (E)$
		\begin{align*}
		\lim_{n \ra \infty} \Vert A_r^{n} - A_r \Vert_{c (E, X) \ra X} = 0 .
		\end{align*}
	Thus, we conclude that
		\begin{align*}
		\Vert A_r \Vert_{c (E, X) \ra X} = \lim_{n \ra \infty} \int_{K_n} |a (r, t)| \, d\mu (t) = \int_E |a (r, t) | \, d\mu (t) .
		\end{align*}
	
	Finally, to prove condition (3), define the function $v_K : E \ra X$ as in \eqref{Eq:FunctionTestCondition3Silverman}, where $K \subset E$ is compact and $x \in X$ with $\Vert x \Vert_X = 1$. Since $v_K \in c (E, X)$ with $\lim_{t \ra \infty} v_K (t) = 0$ and $\cA$ is regular, we conclude that
		\begin{align*}
		\lim_{r \ra \infty} \int_K |a(r, t)| x \, d\mu (t) = 0.
		\end{align*}
	Applying Property \ref{P:BochnerIntegralConstantFunction} from \S\ref{SecSub:BochnerIntegral}, we conclude that condition (3) must be satisfied and this ends the proof.	
	\end{proof}

	\section{Inclusion Between Summability Methods}\label{Sec:InclMethods}
	Regularity for a summability method $\cA = (A, c_A (X), \lim_A )$ can also be expressed in the following way: The domain of the summability method of strong convergence is included in the domain of the application $A$ and the restrictions of the operators $\lim$ and $\lim_A$ to the space $c_\infty (E, X)$ are equal. We generalize this idea of inclusion to two arbitrary summability methods.  
	
	\begin{definition}
	Let $\cA := (A, c_A (X), \lim_A )$ and $\cB := (B, c_B (X) , \lim_B )$ be two summability methods and let $M \subset \cF (E, X)$.
		\begin{enumerate} 
		\item The method $\cA$ is \textit{$M$-included} in the method $\cB$ if the following conditions are satisfied:
			\begin{enumerate}
			\item $M \cap D_A \neq \emptyset$;
			\item $M \cap D_A \subset D_B$;
			\item $M \cap c_A (X) \subseteq c_B (X)$ and $\lim_A v(t) = \lim_B v(t)$ for any $v \in c_A (X)$
			\end{enumerate}
		We denote this by $A \subseteq_M B$. 
		\item The methods $\cA$ and $\cB$ are \textit{$M$-equivalent} if $A \subseteq_M B$ and $B \subseteq_M A$. In this case, we write $A \sim_M B$.
		\end{enumerate}
	\end{definition}
	When $c_A (X) \subset M$, then $\cA$ is $M$-included in $\cB$ becomes simply $\cA$ is \textit{included} in the method $\cB$ and we denote this by $A \subseteq B$. Also, when $c_A (X) \cup c_B (X) \subset M$, then $\cA$ and $\cB$ are $M$-equivalent becomes $\cA$ and $\cB$ are \textit{equivalent} and we write $A \sim B$. 
	
	For $X = \bC$, the Ces{\`a}ro summability method $(C, c_C (\bC ), \lim_C)$ is included in the Abel summability method $(A^0, c_{A^0} (\bC ), \lim_{A^0})$. A reference for this result is \cite[Theorem 3.6.11(a)]{Boos2000}. The method of proof can be easily adapted to sequences in a Banach space. Therefore, we can claim that the Ces{\`a}ro summability method $(C, c_C (X), \lim_C)$ is included in the Abel summability method $(A^0 , c_{A^0} (X), \lim_{A^0})$. 
	
	For $X = \bC$, the Ces{\`a}ro summability method is equivalent to the Riesz summability method. We won't elaborate here on the Riesz summability method, but for more details, we refer to \cite{Kuttner1962} and \cite{Riesz1924}.
	
	To deduce the Silverman-Toeplitz Theorem for the matrix-summability methods, we used the fact that the original Silverman-Toeplitz Theorem was true for matrix-summability methods applied to complex-valued sequences. Therefore, it seems that when a summability method is included into another for complex-valued sequences, then the first method is included in the second one for vector-valued sequences. This depends, however, strongly on the proof of the inclusion of one summability method into another one for scalar-valued sequences. In \cite{MashreghiPariseRansford2021}, the authors found a way to generalize a result about inclusion of summability methods valid for complex-valued sequences to vector-valued sequences. Before recalling this result, we introduce some vocabulary.
	
	We denote by $M_x := \{ \lambda x \, : \, \lambda \in \bC \}$, where $x \in X$ is a unit vector. The set of complex numbers is isomorphic to $M_x$ and the set $\cF (E, M_x )$ is called the \textit{space of scalar-valued functions on $E$}. This avoids inconsistencies in the next definition. 
	\begin{definition}
	Let $\cA := (A, c_A (X), \lim_A )$ and $\cB := (B, c_B (X) , \lim_B )$ be two summability methods. We say that $\cA$ is \textit{scalar-included} in $\cB$, denoted by $A \subseteq_{\bC} B$ if there exists a unit vector $x \in X$ such that $\cA$ is $\cF (E, M_x)$-included in $\cB$. 
	\end{definition}
	The above definition is simply the definition of $M$-included with $M = \cF (E, M_x)$. In the rest of this paper, since the kernel-summability methods are determined by complex-valued kernels, we will identify $M_x$ with $\bC$. We can now state the result obtained in \cite{MashreghiPariseRansford2021} for sequence-to-function summability methods.
	\begin{theorem}[{\cite[Theorem 5.1]{MashreghiPariseRansford2021}}]
	Let $\cA$ and $\cB$ be two regular sequence-to-function summability methods. Let $X$ and $Y$ be Banach spaces, and let $S : X \ra Y$ and $S_n : X \ra Y$ ($n \geq 0$) be bounded linear operators. Suppose that:
		\begin{enumerate}
		\item $S_n (x) \ra S(x)$ for all $x \in W$, where $W$ is a dense subset of $X$;
		\item $(S_n (X))_{n \geq 0}$ is $\cA$-summable to $S(x)$ for all $x \in X$.
		\item $\cA$ is scalar-included in $\cB$.
		\end{enumerate}
	Then $(S_n (x))_{n \geq 0}$ is $\cB$-summable to $S(x)$ for all $x \in X$.
	\end{theorem}
	\noindent We will recall an important step in the proof of this result to clarify why we are using the weak integral over the Bochner integral in the definition of a kernel-summability method. Denote by $b_n (r)$ the function defining the summability method $\cB$ in the statement of the theorem. The proof is split in two parts: (a) to prove that the series $\sum_{n \geq 0} b_n (r) S_n (x)$ exists and (b) to prove that, as $r \ra \infty$, $\sum_{n \geq 0} b_n (r) S_n (x)$ converges to $S(x)$. The reason to use the weak integral over the Bochner integral comes from the proof of part (a). In Part (a), to prove that the series is convergent, it is shown that the sequence of partial sums $( \sum_{n = 0}^N b_n (r) S_n (x) )_{N \geq 0}$ is Cauchy in $Y$. The limit $y$ of this Cauchy sequence satisfies the properties of the weak integral of the function $n \mapsto b_n (r) S_n (x)$ on $E$, that is:
		\begin{align*}
		\phi (y) = \sum_{n \geq 0} b_n (r) \phi (S_n (x)) .
		\end{align*}
	Unfortunately, the proof cannot be improved to show that the series $\sum_{n \geq 0} b_n (r) S_n (x)$ is in $\ell^1 (\bN, Y) := L^1 (\bN , Y )$, which would be the Bochner integral of the function $n \mapsto b_n (r) S_n (x)$. This is why the weak integral is used in the definition of the domain of a kernel-summability method.
	
	We now generalize the last result to kernel-summability methods. Recall that $E$ is $\sigma$-compact. We will also assume that $F$ is metrizable and its Alexandrov compactification \cite{Mandelkern1989} denoted by $F_\infty$ is metrizable. An example of an $F$ with these properties is an interval $[0, R)$ for some $R \in (0, \infty ]$. 
	\begin{theorem}\label{Thm:inclusionResult}
	Let $X$ and $Y$ be two Banach spaces. Let $\cA = (A, c_A (Y) , \lim_A )$ and $\cB = (B, c_B (Y) , \lim_B )$ be two kernel-summability methods. Let $S_t: X \ra Y$, $S : X \ra Y$ be bounded linear operators for each $t \in E$. Assume that
		\begin{enumerate}
		\item $t \mapsto S_t (x)$ is continuous for any $x \in X$;
		\item $\lim_{t \ra \infty} S_t (w) = S(w)$ for any $w$ in a dense subset $W$ of $X$;
		\item $t \mapsto S_t(x)$ is $\cA$-summable to $S(x)$ for any $x \in X$;
		\item $\cB$ is regular;
		\item $\cA$ is scalar-included in $\cB$.
		\end{enumerate}
	Then, for any $x \in X$, the function $t \mapsto S_t (x)$ is $\cB$-summable to $S(x)$. 
	\end{theorem}
	\begin{proof}
	Let $(K_n)_{n \geq 0}$ be the sequence of compact sets such that $K_n \subsetneq K_{n + 1}$ and $E = \cup_n K_n$. Let $a : F \times E \ra \bC$ and $b : F \times E \ra \bC$ be the two $F\times E$-measurable summability kernels associated to the methods $\cA$ and $\cB$ respectively.
	
	To show that $t \mapsto S_t (x)$ is $\cB$-summable for any $x \in X$, we have to show two things:
		\begin{enumerate}
		\item[(a)] the function $t \mapsto S_t (x) \in D_B$, where $D_B$ is the domain of $\cB$.
		\item[(b)] the function $t \mapsto S_t (x) \in c_B (Y)$ and
			\begin{align*}
			\lim_B S_t (x) = \lim_A S_t (x) .
			\end{align*}
		\end{enumerate}
	
	We first show part (a). We therefore have to show that the weak integral of $t \mapsto b(r, t) S_t (x)$ exists for every $r \in F$ and every $x \in X$. Let $r \in F$ and $x \in X$ be fixed. Since the map $t \mapsto S_t (x)$ is $\cA$-convergent by condition (3), it implies that the weak integral
		\begin{align*}
		A_r (x) := \int_E a(r, t) S_t (x) \, d\mu (t)
		\end{align*}
	exists. In particular, the map $t \mapsto a(r, t) \phi (S_t (x))$ is integrable for every $\phi \in Y^\ast$ and
		\begin{align*}
		\phi (A_r (x)) = \int_E a(r, t) \phi (S_t (x)) \, d\mu (t) .
		\end{align*}
	Since $\cA$ is scalar-included in $\cB$, this means, in particular, that
		\begin{align*}
		\int_E b(r, t) \phi (S_t (x)) \, d\mu (t)
		\end{align*}
	exists. From Proposition \ref{Prop:ContinuousImpliesStronglyMeas}, the function $t \mapsto S_t (x)$ is strongly measurable. Also, since $\cB$ is regular, from Theorem \ref{Thm:SilvermanToeplitzKernel}, the map $t \mapsto b(r, t)$ belongs to $L^1_{\bC} (E)$. Therefore, the map $t \mapsto b(r, t) S_t (x)$ is strongly measurable. Also, by continuity, the map $t \mapsto S_t (x)$ is bounded on each $K_n$ and therefore the restriction of the map $t \mapsto a (r, t) S_t (x)$ to the compact set $K_n$ is Bochner integrable by Property \eqref{P:BochnerEquivalentNormIntegrable} in \S\ref{SecSub:BochnerIntegral}, that is the Bochner integral
		\begin{align*}
		\int_{K_n} b(r, t) S_t (x) \, d\mu (t)
		\end{align*}
	exists. By the Lebesgue Dominated Convergence Theorem (for $L^1_\bC (E)$)
		\begin{align*}
		\int_E b(r, t) \phi (S_t (x)) \, d\mu (t) = \lim_{n \ra \infty} \int_{K_n} b(r, t) \phi (S_t (x)) \, d\mu (t) = \lim_{n \ra \infty} \phi \Big( \int_{K_n} b(r, t) S_t (x) \, d\mu (t) \Big) .
		\end{align*}
	and, in particular,
		\begin{align*}
		\sup_{n \geq 0} \Big| \phi \Big( \int_{K_n} b(t, r) S_t (x) \, d\mu (t) \Big) \Big| < \infty .
		\end{align*}
	This is true for any $\phi \in Y^\ast$. Therefore, by the Banach-Steinhauss Theorem and the Hahn-Banach Theorem,
		\begin{align}
		\sup_{n \geq 0} \Big\Vert \int_{K_n} b(r, t) S_t (x) \, d\mu (t) \Big\Vert_Y < \infty . \label{Eq:PointwiseBoundedIntegrals}
		\end{align}
	For any integer $n \geq 0$, define the linear operator $B_{r,n} : X \ra Y$ by
		\begin{align*}
		B_{r, n} (x) := \int_{K_n} b(r, t) S_t (x) \, d\mu (t) .
		\end{align*}
	Then $B_{r, n}$ is well-defined and is continuous. Indeed, by the assumption (1) and the fact that $K_n$ is compact, for any $x \in X$,
		\begin{align*}
		\sup_{t \in K_n} \Vert S_t (x) \Vert_Y < \infty .
		\end{align*}
	Therefore, the Banach-Steinhauss Theorem gives
		\begin{align*}
		M_n := \sup_{t \in K_n} \Vert S_t \Vert_{X \ra Y} < \infty .
		\end{align*}
	So, if $x \in X$, then
		\begin{align*}
		\Vert B_{r, n} (x) \Vert_Y \leq \int_{K_n} |b(r, t)| \Vert S_t (x) \Vert_Y \, d\mu (t) \leq M_n \Big( \int_E |b (r, t) | \, d\mu (t) \Big) \Vert x \Vert_X .
		\end{align*}
	Now, using \eqref{Eq:PointwiseBoundedIntegrals} and a third time the Banach-Steinhauss Theorem, we obtain
		\begin{align}
		M := \sup_{n \geq 0} \Vert B_{n, r} \Vert_{X \ra Y} < \infty . \label{Eq:UniformBoundIntegralOperators}
		\end{align}
		
	Let $\varepsilon > 0$. Use the density of $W$ in assumption (2) to choose a $w \in W$ such that $\Vert x - w \Vert_X < \varepsilon$ and such that $S_t (w) \ra S(w)$ as $t \ra \infty$. In particular, the map $t \mapsto S_t (w)$ is bounded on $E$. This last fact combined with the fact that the map $t \mapsto b(r, t)$ is integrable implies that the function $t \mapsto b(r, t) S_t (w)$ is Bochner integrable on $E$. Therefore, the sequence $(B_{r, n} (w))_{n \geq 0}$ is a Cauchy sequence in $Y$. This means that there is an $N$ such that whenever $m , n \geq N$,
		\begin{align*}
		\Vert B_{r, m} (w) - B_{r, n} (w) \Vert_Y < \varepsilon .
		\end{align*}
	Therefore, we have
		\begin{align*}
		\Vert B_{r, m} (x) - B_{r, n} (x) \Vert_Y & \leq 2 M \Vert x - w \Vert_X + \Vert B_{r, m} (w) - B_{r, n} (w) \Vert_Y \\
		& \leq (2 M + 1 ) \varepsilon .
		\end{align*}
	From this, we conclude that $(B_{r, n}(x))_{n \geq 0}$ is a Cauchy sequence and since $Y$ is complete, $\lim_{n \ra \infty} B_{r, n} (x)$ exists, say $I^B (x)$. The vector $I^B (x)$ is the weak integral of the map $t \mapsto b(r, t) S_t (x)$ because for any $\phi \in Y^\ast$, we have
		\begin{align*}
		\phi (I^B (x)) = \lim_{n \ra \infty} \int_{K_n} \phi (b(r, t) S_t (x)) \, d\mu (t) = \int_E \phi (b (r, t) S_t (x) ) \, d\mu (t ). 
		\end{align*}
	This concludes the proof of part (a).
	
	Let's turn our attention to part (b). Fix $x \in X$. We will show that
		\begin{align*}
		\lim_{r \ra \infty} \int_E b(r, t) S_t (x) \, d\mu (t) = \lim_{r \ra \infty} \int_E a(r, t) S_t (x) \, d\mu (t) = S(x) .
		\end{align*}	
		
	For $r \in F$, define the linear operator $B_r : X \ra Y$ by
		\begin{align*}
		B_r (x) = \int_E b(r, t) S_t (x) \, d\mu (t) ,
		\end{align*}		 
	where the integral on the right-hand side is the weak integral $I^B (x)$ from part (a). The linear operator $B_r$ is also continuous because $B_r (x) = \lim_{n \ra \infty} B_{r, n} (x)$, where $B_{r, n}$ are the continuous linear operators defined in part (a). Using the characterization of limits in the metric space $F_\infty$ in terms of sequences, let $r_n \ra \infty$ be an arbitrary sequence converging to $\infty$ as $n \ra \infty$. As a consequence of assumption (3), we have that $\phi (A_r(x)) \ra \phi (S(x))$ as $r \ra \infty$ for any $\phi \in Y^\ast$. Therefore, $\phi (B_{r_n} (x)) \ra \phi (S(x))$ as $n \ra \infty$ and, in particular,
		\begin{align*}
		\sup_{n \geq 0} | \phi (B_{r_n} (x)) | < \infty .
		\end{align*}
	Using the Banach-Steinhauss Theorem twice, we conclude that
		\begin{align*}
		M' := \sup_{n \geq 0} \Vert B_{r_n} \Vert_{X \ra Y} < \infty .
		\end{align*}
	Let $\varepsilon > 0$. Choose $w \in W$ so that $\Vert x - w \Vert_X < \varepsilon$ and $S_t (w) \ra S(w)$. Since $\cB$ is regular, $B_r (w) \ra S(w)$, as $r \ra \infty$. Therefore, choose $N$ such that $\Vert B_{r_n} (w) - S(w) \Vert_Y < \varepsilon$ for any $n \geq N$. We then have, for $n \geq N$, 
		\begin{align*}
		\Vert B_{r_n} (x) - S(x) \Vert_Y & \leq \Vert B_{r_n} (x - w) \Vert_Y + \Vert S(x - w) \Vert_Y + \Vert B_{r_n} (w) - S(w) \Vert_Y \\
		& \leq \Vert B_{r_n} \Vert_{X \ra Y} \Vert x - w \Vert_X + \Vert S \Vert_{X \ra Y} \Vert x - w \Vert_X + \varepsilon \\
		& \leq (M' + \Vert S \Vert_{X \ra Y} + 1) \varepsilon .
		\end{align*}
	We conclude that $\lim_{n \ra \infty} B_{r_n} (x) = S(x)$. Since the sequence $(r_n)_{n \geq 0}$ was arbitrary, we conclude that $\lim_{r \ra \infty} B_r (x) = S(x)$ as required. This ends the proof.	
	\end{proof}
	If we put $E := \bN$ with the discrete topology, then we recover Theorem 5.1 in \cite{MashreghiPariseRansford2021}. This comes from the fact that the first assumption is automatically satisfied when $E = \bN$ and the fourth assumption is included in the prelude of the Theorem. Notice that in the above result, we didn't assume $\cA$ to be regular, in contrast to what was assumed in Theorem 5.1 in \cite{MashreghiPariseRansford2021}. Note also that we assumed implicitly that the methods $\cA$ and $\cB$ have the same parameter space $F$. In fact, the kernel functions defining $\cA$ and $\cB$ can have different parameters space in their first argument.
	
	Furthermore, the last result requires the existence of a dense subspace $W$ and uses a function with values in $B (X, Y)$, the Banach space of bounded linear operators from $X$ into $Y$. We might ask if those assumptions can be dropped in the statement. As a careful inspection of the above proof shows, one will finds out that the hypothesis (2) is crucial. Indeed, in the proof, we need this assumption to show that the weak integral defining $B_r(x)$ exists for each $r \in F$. Otherwise, the best conclusion that we can draw from this hypothesis is that the sequence $(B_{r, n} (x))_{n \geq 0}$ is weakly Cauchy in $Y$. In this situation, unfortunatly, the weak limit might not exist. Examples of Banach spaces $Y$ where weakly Cauchy sequences don't have limit are $L^\infty ([0, 1])$ and $A (\bD )$ (the disk Algebra). For reference, see \cite[Example 4.5]{Conway1990}. However, if more restrictions are added on the Banach space $X$, we can obtain the following result.
	\begin{theorem}\label{Thm:WeakInclusion}
	Let $X$ be a Banach space. Let $\cA$, $\cB$ be two kernel-summability methods. Assume further that
		\begin{enumerate}
		\item $\cA$ is scalar-included in $\cB$;
		\item $\cB$ is regular;
		\item $X$ is reflexive.
		\end{enumerate}
	Then the method $\cA$ is \textbf{weakly-included} in the method $\cB$.
	\end{theorem}
	In the above statement, a summability method $\cA$ is weakly included in another summability method $\cB$ if the following conditions are satisfied:
		\begin{itemize}
		\item[(a)] $D_A \subset D_B$;
		\item[(b)] If $t \mapsto v(t)$ is $\cA$-summable weakly to $x \in X$ as $t \ra \infty$, then $t \mapsto v(t)$ is $\cB$-summable weakly to $x$.
		\end{itemize}
	
	\begin{proof}[Proof of Theorem \ref{Thm:WeakInclusion}]
	Let $(K_n)$, $F_\infty$, $a, b : F \times E \ra \bC$ be as in the preceding proof. To show that $\cA$ is weakly-included in the method $\cB$, we have to show that
		\begin{enumerate}
		\item[(a)] $D_A \subset D_B$;
		\item[(b)] If $\lim_A \phi (v(t)) = \phi (x)$ for every $\phi \in X^\ast$, then $\lim_B \phi (v(t)) = \phi (x)$.
		\end{enumerate}
	
	We first prove part (a). Let $v \in D_A$. Then, this means that 
		\begin{align*}
		A_r (v) = \int_E a(r, t) v(t) \, d\mu (t)
		\end{align*}
	exists for every $r \in F$. From the properties of the integral, if $\phi \in X^\ast$, then
		\begin{align*}
		\phi (A_r (v)) = \int_E a(r, t) \phi (v(t)) \, d\mu (t) .
		\end{align*}
	Since $\cA$ is scalar-included in $\cB$, we find, in particular, that the map $t \mapsto b(r, t) \phi (v(t))$ is integrable for every $r \in F$ and every $\phi \in X^\ast$.
	
	Fix $r \in F$ for the moment. Define the linear operator $T_r : X^\ast \ra L^1_{\bC} (E)$ by
		\begin{align*}
		T_r (\phi ) := b(r, t) \phi (v(t)) \quad (\phi \in X^\ast ) .
		\end{align*}
	By the closed-graph Theorem, the operator $T_r$ is bounded. Its adjoint is then also defined and $T_r^\ast : (L^1_{\bC} (E))^\ast \ra X^{\ast\ast}$. Since $E$ is a $\sigma$-finite measure space, its dual is isomorphic to $L^\infty_{\bC} (E)$ and the isomorphism is giving by
		\begin{align*}
		\psi (u) := \int_E u(t) w(t) \, d\mu (t) \quad (u \in L^1_\bC (E) , \, w \in L^\infty_\bC (E) )
		\end{align*}
	with $\psi$ in the dual of $L^1_\bC (E)$ (see \cite[Theorem 6.16]{RudinRCA1987}). The space $X$ is also reflexive by assumption (3) and therefore $T_r^\ast : L^\infty_\bC (E) \ra X$, under the appropriate identification. %Denote by $J : X \ra X^{\ast\ast}$ the canonical embedding of $X$ in $X^{\ast\ast}$ defined by $J(x) := \hat{x}$, where $\hat{x} : X^\ast \ra \bC$ is defined by $\hat{x} (\phi) = \phi (x)$, for any $\phi \in X^\ast$. so $T_r^\ast : L^\infty_\bC (E) \ra X$. 
	The vector $T_r^\ast (\chi_E )$ is then the weak integral of $t \mapsto b(r, t) v(t)$ as one can check using the isomorphisms of $(L^1_\bC (E))^\ast$ with $L^\infty_\bC (E)$ and of $X$ with $X^{\ast\ast}$. 
	
	The part (b) is immediate from the assumption (1).
	\end{proof}
	The above result may be sufficient in certain applications. For example, it can be applied to the convergence/divergence of Taylor series of holomorphic functions in Hilbert spaces as in \cite{MashreghiPariseRansford2021b, MashreghiPariseRansford2021}, instead of the Theorem \ref{Thm:inclusionResult}.
	
	\section{Consequences On Summability In Holomorphic Banach Spaces}\label{Sec:ConsHolBanach}

Recall that a Banach space $X \subset \hol (\bD )$ is a \textit{Banach holomorphic function space on the unit disk} if the inclusion map $\mathfrak{i} : X \ra \hol (\bD )$ is continuous. We will use the abbreviation BHFD to refer to these spaces. Examples of BHFD are the Hardy spaces $H^p$ ($0 < p \leq \infty$), the disk algebra, the Wiener algebra, the weighted Dirichlet spaces, Bloch spaces, the Bergman spaces $A^p$ ($0 < p < \infty$), the de Branges-Rovnyak spaces, and more. 

Suppose that the set of polynomials is dense in $X$, so in particular any polynomial is in $X$. We define maps $S_n : X \ra X$ by
	\begin{align*}
	S_n (f) (z) := \sum_{k = 0}^n a_k z^k \quad (f \in X, \, z \in \bD ). 
	\end{align*}
Each $S_n$ maps a function $f(z) = \sum_{m = 0}^\infty a_m z^m$ in $X$ to the $n$-th partial sums of the power series of $f$. Given a summability method $\cA = (A, c_A (X) , \lim_A)$, the Taylor series of a function $f \in X$ is $\cA$-summable to some $g \in X$ if the sequence $(S_n(f))_{n \geq 0}$ is $\cA$-summable to $g$.

\begin{corollary}\label{Cor:AppliedTalorSums}
Let $X$ be an BHFD and assume the set of polynomials is dense in $X$. Let $\cA = (A, c_A (X) , \lim_A )$ and $\cB = (B, c_B (X) , \lim_B)$ be two sequence-to-function summability methods such that $\cA$ is scalar-included in $\cB$ and $\cB$ is regular. If the Taylor series is $\cA$-summable to $f$ for every $f \in X$, then it is also $\cB$-summable to $f$ for every $f \in X$.
\end{corollary}
\begin{proof}
Let $E = \bN$ and let $F$ be as in \S\ref{Sec:VecValueInt}. To obtain the result, we only need to verify the assumptions of Theorem \ref{Thm:inclusionResult}. 

Let $W$ be the set of polynomials. We know that $W$ is dense in $X$. Define $S(f) = f$, the identity operator on $X$. We first show that $S_n : X \ra X$ is continuous for any $n \geq 0$. Fix $n \geq 0$ and let $f \in X$ with $f(z) = \sum_{m \geq 0} a_m z^m$. Since $a_m = \frac{f^{(m)}(0)}{m!}$, we have
	\begin{align*}
	\Vert S_n (f) \Vert_X \leq \sum_{k = 0}^n \Big| \frac{f^{(n)} (0)}{n!} \Big| \Vert z^k \Vert_X .
	\end{align*}
The maps $D_k : \hol (\bD ) \ra \bC$ defined by $D_k (f) := f^{(k)} (0)$ is continuous and linear. Since the inclusion $\mathfrak{i} : X \ra \hol (\bD )$ is continuous, the linear functionals $\left. D_k \right|_X : X \ra \bC$ is continuous on $X$. Therefore there are constants $C_k$ ($k \geq 0$) such that
	\begin{align*}
	|f^{(k)} (0)| \leq C_k \Vert f \Vert_X .
	\end{align*}
From this last estimate, we obtain the following estimate:
	\begin{align*}
	\Vert S_n (f) \Vert_X \leq \Big( \sum_{k = 0}^n \frac{C_k}{k!} \Vert z^k \Vert_X \Big) \Vert f \Vert_X
	\end{align*}
and $S_n : X \ra X$ is continuous.

Secondly, for each polynomial $p$ with $N := \mathrm{deg}\, p$, we have $S_n (p) = p$, when $n \geq N$. Therefore $\lim_{n \ra \infty} S_n (p) = p$ for any polynomial $p$. Finally, the method $\cA$ is scalar-included in the method $\cB$. Assumptions (1), (2), (4) and (5) in Theorem \ref{Thm:inclusionResult} are then satisfied.

Suppose that the Taylor series of any $f \in X$ is $\cA$-summable to $f$. Then the assumption (4) of Theorem \ref{Thm:inclusionResult} is satisfied and we can conclude that the Taylor series of any $f \in X$ is $\cB$-summable to $f$.
\end{proof}

We also have a similar result for the Abel means (or the radial dilates) modulo a certain additional assumption. For $0 < r < 1$, let $A_r : \hol (\bD ) \ra \hol (\bD )$ be defined as followed:
	\begin{align*}
	A_r (f) := (1-r) \sum_{m \geq 0} S_m (f) r^n \quad (f \in X ).
	\end{align*}
This is a continuous linear map on $X$ if we assume that $\limsup_{k \ra \infty} \Vert z^k \Vert_X^{1/k} \leq 1$. We say $f$ is $\cB$-Abel summable to some $g \in X$ if $A_r(f)$ is $\cB$-summable to $g$.

\begin{corollary}\label{Cor:AbelSummabilityPowerSeries}
Let $X$ be an BHFD and assume the set of polynomials is dense in $X$ with $\limsup_{k \ra \infty} \Vert z^k \Vert_X \leq 1$. Let $\cB = (B, c_B (X) , \lim_B )$ and $\mathcal{D} = (D, c_D (X) , \lim_D )$ be two kernel-summability methods such that $\cB$ is scalar-included in $\cD$ and $\mathcal{D}$ is regular. If any $f \in X$ is $\cB$-Abel summable to $f$, then for any $f \in X$, it is $\cD$-Abel summable to $f$.
\end{corollary}
\begin{proof}
Since $\limsup_{k \ra \infty} \Vert z^k \Vert_X \leq 1$, then each map $A_r : X \ra X$ is a bounded linear operator. The proof then follows the same scheme as the preceding proof.
\end{proof}

Notice that $A_r (f)$ can be replaced by any other summability methods. For example, we can use the logarithmic summability methods and define
	\begin{align*}
	L_r (f) (z) := -\frac{1}{\log (1 - r)} \int_0^r \frac{A_r (f) (z)}{1 - t} \, dt .
	\end{align*}
In Corollary \ref{Cor:AbelSummabilityPowerSeries}, the $\cB$-Abel summability is replaced by the $\cB$-logarithmic summability of the power series of $f$. 
	
\section{Aknowledgement}
The author would like to thank one of the anonynomous referee for his careful reading of the manuscript. His comments greatly improved the presentation of the paper.

The last section (\S\ref{Sec:ConsHolBanach}) is a generalization of a collection of results that appeared in the author's Ph.D. thesis (in French) that wasn't published anywhere. Therefore, I would like to thank Thomas Ransford for pointing out  a mistake in the proof of Corollary \ref{Cor:AppliedTalorSums} while the author was in the process of writing his PhD thesis.

\bibliographystyle{plain}
\bibliography{biblio-All-these.bib}

\begin{thebibliography}{10}

\bibitem{Boos2000}
J.~Boos.
\newblock {\em Classical and {M}odern {M}ethods in {S}ummability}.
\newblock Oxford Mathematical Monographs. Oxford University Press, Oxford,
  2000.
\newblock Assisted by Peter Cass, Oxford Science Publications.

\bibitem{Borwein1957a}
D.~Borwein.
\newblock On a scale of {A}bel-type summability methods.
\newblock {\em Proc. Cambridge Philos. Soc.}, 53:318--322, 1957.

\bibitem{Conway1990}
J.~B. Conway.
\newblock {\em A {C}ourse in {F}unctional {A}nalysis}, volume~96 of {\em
  Graduate Texts in Mathematics}.
\newblock Springer-Verlag, New York, second edition, 1990.

\bibitem{hardy1949}
G.~H. Hardy.
\newblock {\em Divergent {S}eries}.
\newblock Oxford, at the Clarendon Press, 1949.

\bibitem{Kangro1976}
G.~F. Kangro.
\newblock Theory of summability of sequences and series.
\newblock {\em J. Math. Sci.}, 5:1--45, 1976.

\bibitem{Kuttner1962}
B.~Kuttner.
\newblock {On Discontinuous Riesz Means of Type n}.
\newblock {\em Journal of the London Mathematical Society}, s1-37(1):354--364,
  1962.

\bibitem{leonetti2020}
Paolo Leonetti.
\newblock A characterization of cesàro convergence.
\newblock {\em The American Mathematical Monthly}, 128(6):559--562, 2021.

\bibitem{Leonetti2022}
Paolo Leonetti.
\newblock Regular matrices of unbounded linear operators, 2022.

\bibitem{Lewis2022}
Andrew~D. Lewis.
\newblock {Integrable and absolutely continuous vector-valued functions}.
\newblock {\em Rocky Mountain Journal of Mathematics}, 52(3):925 -- 947, 2022.

\bibitem{Leon2020}
F.~León-Saavedra, M.~del~P. Romero de~la Rosa, and A.~Sala.
\newblock Schur lemma and uniform convergence of series through convergence
  methods.
\newblock {\em Mathematics}, 8(10), 2020.

\bibitem{Maddox1980}
I.~J. Maddox.
\newblock {\em Infinite {M}atrices of {O}perators}, volume 786 of {\em Lecture
  Notes in Mathematics}.
\newblock Springer, Berlin, 1980.

\bibitem{Mandelkern1989}
Mark Mandelkern.
\newblock Metrization of the one-point compactification.
\newblock {\em Proc. Amer. Math. Soc.}, 107(4):1111--1115, 1989.

\bibitem{MashreghiPariseRansford2021b}
Javad Mashreghi, Pierre-Olivier Paris\'{e}, and Thomas Ransford.
\newblock Ces\`aro summability of {T}aylor series in weighted {D}irichlet
  spaces.
\newblock {\em Complex Anal. Oper. Theory}, 15(1):Paper No. 7, 8, 2021.

\bibitem{MashreghiPariseRansford2021}
Javad Mashreghi, Pierre-Olivier Paris\'{e}, and Thomas Ransford.
\newblock Power-series summability methods in de {B}ranges--{R}ovnyak spaces.
\newblock {\em Integral Equations Operator Theory}, 94(2):Paper No. 20, 17,
  2022.

\bibitem{Riesz1924}
M.~Riesz.
\newblock Sur l'equivalence de cértaines methodes de sommation.
\newblock {\em Proceedings of the London Mathematical Society},
  s2-22(1):412--419, 1924.

\bibitem{Robinson1950}
A.~Robinson.
\newblock On functional transformations and summability.
\newblock {\em Proc. London Math. Soc. (2)}, 52:132--160, 1950.

\bibitem{RudinRCA1987}
W.~Rudin.
\newblock {\em Real and Complex Analysis, 3rd Ed.}
\newblock McGraw-Hill, Inc., 1987.

\end{thebibliography}

\end{document}